\documentclass[11pt,oneside, reqno]{amsart} 
\usepackage[foot]{amsaddr}

\usepackage[utf8x]{inputenc} 
\usepackage{tcolorbox}

\usepackage{geometry} 
\usepackage{graphicx} 
\usepackage{color}
\usepackage{stmaryrd}

\usepackage{array} 
\usepackage{paralist} 
\usepackage{verbatim} 
\usepackage{subfig} 

\usepackage{amsfonts}
\usepackage{dsfont}
\usepackage{amsmath}
\usepackage{amsthm}
\usepackage{amssymb}
\usepackage{mathtools}

\usepackage{amsthm}
\theoremstyle{definition}

\newtheorem{definition}{Definition}
\newtheorem{lemma}{Lemma}
\newtheorem{remark}{Remark}
\newtheorem{example}{Example}
\newtheorem{example*}{Example\textsuperscript{*}}
\newtheorem{proposition*}{Proposition\textsuperscript{*}}
\newtheorem{corollary}{Corollary}
\newtheorem{corollary*}{Corollary\textsuperscript{*}}
\newtheorem{theorem}{Theorem}
\newtheorem{proposition}{Proposition}

\newcommand{\Rmax}{R_{\mathrm{max}}}
\newcommand{\Rmin}{R_{\mathrm{min}}}

\newcommand{\We}{\mathrm{We}}

\def\Limes#1#2 {\lim\limits_{#1\rightarrow #2}}

\def\eps{\epsilon}

\def\R{\mathbb{R}}

\def\XXint#1#2#3{{\setbox0=\hbox{$#1{#2#3}{\int}$ }
\vcenter{\hbox{$#2#3$ }}\kern-.59\wd0}}

\newcommand{\grad}{\nabla}

\renewcommand{\div}{\operatorname{div}}

\newcommand{\dd}{\, \mathrm{d}}

\renewcommand{\div}{\grad\cdot}
\renewcommand{\epsilon}{\varepsilon}

\renewcommand{\&}{and}

\usepackage[square]{natbib}
\setcitestyle{numbers}

\usepackage{verbatim}
\usepackage{hyperref}

\usepackage{xcolor}

\begin{document}

\title{Non-Existence of thick bubble rings at low Weber numbers}
\thanks{This work is funded by the Deutsche Forschungsgemeinschaft (DFG, German Research Foundation) under Germany's Excellence Strategy EXC 2044/2 - 390685587, Mathematics M\"unster: Dynamics-Geometry-Structure, and  grant  531098047.}

\author{Yuanjiang Han and Christian Seis}

\address{Institut f\"ur Analysis und Numerik,  Universit\"at M\"unster, Orl\'eans-Ring 10, 48149 M\"unster, Germany.}
\email{yhan@uni-muenster.de, seis@uni-muenster.de}

\begin{abstract}
We examine the existence of thick bubble rings within the framework of the free-boundary capillary Euler equations, focusing on the regime of low Weber numbers. Although spheroidal bubbles are known to approach a spherical shape in this limit, the possibility of thick bubble rings persisting at low Weber numbers has remained uncertain. In contrast to the ordinary Euler equations, which admit thick vortex ring solutions, our analysis reveals that the free-boundary capillary Euler equations do not support thick bubble rings at low Weber numbers. This distinction highlights the significant impact of surface tension on the behavior of vortex rings in the capillary regime.
\end{abstract}

\maketitle

\section{Introduction}
A vortex ring in an ideal fluid is a three-dimensional toroidal region in which the vorticity is concentrated. Typical examples include smoke rings, such as those deliberately exhaled by tobacco smokers or those that appear naturally during volcanic eruptions. More recently, underwater vortex rings generated by cephalopods through jet propulsion have attracted interest from engineers and zoologists; see, for example, \cite{LuoHouXuCaoPan23} and the references therein.

Mathematically, classical vortex rings arise as travelling-wave solutions of the axisymmetric Euler equations. These are steady in the moving frame. The first example, and the only one that is completely explicit, was given by Hill in 1894 \cite{Hill1894}; despite being called a "ring", Hill's solution has, in fact, a perfectly spherical shape, with only the inner stream surfaces being toroidal. A major breakthrough came in 1969, when Fraenkel constructed  thin vortex rings whose minor (i.e., core) radii are much smaller than their major radii (i.e., the distances of the vortex core to the axis) \cite{Fraenkel70}. Thick vortex rings, whose major radii are comparable in size to the core radii, where derived slightly later by  Norbury \cite{Norbury72}. It is interesting to note that  for a given vortex intensity or circulation, Hill's spherical vortex is the only possible simply-connected vortex \cite{AmickFraenkel86}, and any thick or thin ring must coincide with Norbury's or Fraenkel's solutions, respectively \cite{AmickFraenkel88,CaoLaiQinZhanZou22,CaoQinYuZhanZou22+}. A wealth of further constructions can be found in the literature, for example in \cite{FraenkelBerger74,BergerFraenkel80,FriedmanTurkington81,AmbrosettiStruwe89,Yang95,BadianiBurton01,Burton03,CaoQinYuZhanZou22+}. Non-steady vortex rings were analysed, for instance, in \cite{BenedettoCagliotiMarchioro00,JerrardSeis17}.

In the present paper, we focus on hollow vortex rings, where the fluid circulates around a toroidal cavity. Examples of such vortex rings include air bubble rings formed underwater by cetaceans or scuba divers. To accurately model these bubble  rings, we assume that the pressure on the rings' surfaces is determined by surface tension via the Young--Laplace law. The simplest mathematical model for bubble rings is the free-boundary capillary Euler equation. This model involves a single physical non-dimensional constant, the Weber number $\We$, which quantifies the ratio of inertial forces to surface tension forces, see equation \eqref{2} below. The Weber number will play a crucial role in our analysis.

To the best of our knowledge, there are currently only two known families of travelling-wave vortex solutions to the free-boundary capillary Euler equations. These comprise the thin toroidal bubble rings constructed by Meyer and the second author \cite{MeyerSeis26}, and the spheroidal bubbles studied by Meyer, Niebel, and the second author \cite{MeyerNiebelSeis25}. Notably, the spheroidal bubbles approximate a spherical shape in the limit of low Weber numbers where surface tension effects dominate, which aligns with basic intuition. Additionally, we mention the rotating capillary vortex sheet solutions found by Murgante, Roulley, and Scrobogna as solutions to the planar Euler equations, which are travelling waves in the azimuthal direction \cite{MurganteRoulleyScrobogna24}.

In view of the examples in \cite{MeyerNiebelSeis25}, it is worth noting that the free-boundary capillary Euler equations exhibit greater flexibility than the ordinary Euler equations with regard to spheroidal solutions, as the former allow for a wider range of steady configurations beyond Hill's spherical vortex, which is the only possible steady, simply-connected vortex configuration in the ordinary Euler equations.

In the present paper, we aim to show that the situation is reversed for toroidal solutions: Indeed, our goal is to show that thick bubble ring solutions cannot exist if the Weber number is sufficiently low, that is, if surface tension effects dominate. This observation is in stark contrast to the ordinary Euler equations, where thick rings are known to exist thanks to Norbury's work, and thus, for the capillary free-boundary Euler equations, a certain rigidity emerges. This interplay between flexibility and rigidity, rooted in the topological properties of the vortex configurations, makes the study of free-boundary capillary Euler equations a rich and complex problem.

Recently, Niebel established a rigidity result for a related problem for the planar capillary free-boundary Euler equation: for sufficiently low Weber numbers, there are no stationary bubble configurations apart from the spherical one \cite{Niebel25}. Niebel's work extends a seminal rigidity result by Serrin \cite{Serrin71} to the capillary setting.

The question if thick bubble ring solutions exist for sufficiently high Weber numbers remains an interesting open problem that we plan to address in a future work.

\section{Mathematical setting and results}
Steady vortex rings are travelling wave solutions, which move at a constant speed \(W\) in a fixed direction, say \(e_z\), without changing their shape. Hence, if $u=u(t,x)\in\R^3$ is the velocity of the ambient fluid, we may write
\[
u(t,x)=U(x-tWe_z).
\]
Since in axisymmetric coordinates,
\[
x=r e_r +z e_z,\quad U=U_r(r,z)e_r+U_z(r,z)e_z, 
\]
the vector field \(rU\) is divergence free, there exists a stream function $\psi$ satisfying \(rU= \grad^{\perp}\psi.\) Here, $\grad=(\partial_r,\partial_z)^T$ is the cylindrical gradient. The co-moving velocity field \(U-We_z\) is tangential to the ring's surface and hence  the modified stream function \(\Psi= \psi-Wr^2/2\) should be constant on the boundary of the ring.

In axisymmetric coordinates, the bubble ring is described by its cross-section $E$ in the meridional half plane $\R^2_+=\{(r,z)\in\R_+\times \R\}$. Its boundary is denoted by \(\partial E\), and we will tacitly assume that it is regular, say $C^2$. 

Considering the interior of a bubble ring as vacuum for simplicity, the stream function $\psi$ satisfies the overdetermined boundary problem
\begin{align}
    -\div\left(\frac1r\grad\psi\right)&=0\quad\mbox{in }\R^2_+\setminus \overline E,\label{4}\\
    \psi -\frac{W}2r^2 &=\gamma\quad\mbox{on }\partial E,\label{3}\\
    2\sigma H - \rho\left(\frac1r\partial_n \psi -W n\cdot e_r\right)^2 + \lambda&=0\quad \mbox{on }\partial E,\label{5}\\
    \psi &=0\quad \mbox{on }\partial\R^2_+,\label{6}\\
   \psi&\to0\quad\mbox{as }|(r,z)|\to\infty.\label{7}
\end{align}
Equation \eqref{4} simply states that the fluid is irrotational outside of the ring, and thus, the vorticity is concentrated on its surface $\partial E$. The constant $\gamma$ in \eqref{3} is called the flux constant. The quantity $H$ in \eqref{5} is  (twice) the mean curvature of the ring. In axisymmetric coordinates, it can be expressed as 
 \begin{equation}
     \label{20}
     H=\kappa+\frac{n\cdot e_r}{r},
 \end{equation}
 where \(\kappa\) is the curvature of the one-dimensional boundary curve \(\partial E\), designed to be positive on convex sets,  and \(n\) is the unit normal vector pointing into the exterior domain  \(\R^2_+\setminus\overline{E}\). The constant $\sigma$ is the surface tension coefficient and $\rho$ is the mass density of the fluid. Equation \eqref{5} has its origin in the Young--Laplace law $P-p=\sigma H$, in which $p$ is the pressure of the ambient fluit and $P$ the constant interior pressure. It can be brought into the present form by expressing the pressure in terms of the kinetic energy density via the Bernoulli equations. We call $\lambda$ accordingly the Bernoulli constant. A detailed derivation of this model can be found, for instance, in Section 1 of \cite{MeyerSeis26}.
 
The problem \eqref{4}, \eqref{3}, \eqref{5}, \eqref{6}, and \eqref{7}  is called \emph{overdetermined} because the solution $\psi$ satisfies both a Dirichlet boundary condition \eqref{3} and a Neumann-type boundary condition \eqref{5}. It is evident that for a given domain $E$, the elliptic problem \eqref{4} cannot be solved with both boundary conditions simultaneously. Instead, in the above problem, besides $\psi$, also the domain $E$ is an unknown of the problem, and so are the constants $\gamma$, $W$ and $\lambda$. Again, we refer to the works \cite{MeyerNiebelSeis25} and \cite{MeyerSeis26} for two examples, in which this overdetermined problem is solved in the case of spheroidal bubbles and thin bubble rings, respectively. Overdetermined boundary problems are classical topics in the theory of partial differential equations and geometric analysis; a nice overview is provided in ~\cite{DominguezVazquezEncisoPeraltaSalas23}.

The strength of the bubble ring is characterised by the circulation,
\begin{equation}
\label{10}
\beta = -\int_{\partial E} \frac1r\partial_n \psi\dd s.
\end{equation}
According to Kelvin's circulation theorem, this quantity is a constant of motion also in the non-steady setting, because the loop $\partial E$ moves with the fluid flow. We shall treat the circulation as a control parameter and prescribe it in the following. In particular, we deduce from \eqref{3} that the speed $W$ of the ring will respond linearly to changes in $\beta$, which is consistent with the Kelvin--Hicks formula for the speed of a thin ring (see, e.g.~(1.19) in \cite{MeyerSeis26}).

By non-dimensionalising equations \eqref{4}, \eqref{3}, \eqref{5}, \eqref{6}, \eqref{7}, and \eqref{10} (see Section \ref{S3}), we find that the problem depends on a few non-dimensional parameters, including the circulation-based \emph{Weber number}
\begin{equation} \label{2}
\We=\frac{\sqrt{2\pi} \rho \beta^2}{\sigma \sqrt{|E|}},
\end{equation}
which quantifies the relative importance of the fluid's inertia, as characterised by the circulation, compared to its surface tension. Here, $|E|$ represents the area of the cross-section. The prefactor $\sqrt{2\pi}$ in the numerator is introduced for convenience as it simplifies our later presentation.

The main goal of this study is to demonstrate the non-existence of thick bubble rings in the context of low Weber numbers, where surface tension effects are dominant. To this end, we assume that solutions to \eqref{4}, \eqref{3}, \eqref{5}, \eqref{6}, \eqref{7}, and \eqref{10} describing thick bubble rings exist, and we focus on deriving bounds on the Weber number for such solutions.

We focus on configurations for which the cross-section $E$ is \emph{convex} and \emph{symmetric} in $z$, in the sense that $(r,z)\in E$ if and only if $(r,-z)\in E$. Both assumptions are typical for vortex rings.  Furthermore, we suppose that the stream function in the moving frame $\Psi=\psi-W r^2/2$, which is constant on the ring's surface \eqref{3}, is maximal on the surface, so that
\begin{equation}
    \label{47}
    \partial_n \Psi \le 0\quad\mbox{on }\partial E.
\end{equation}
This assumption ensures that the fluid circulates in a fixed, namely clockwise, direction around the ring.

In the literature, a bubble ring is called \emph{thick} if its cross-sectional radius $a$ is comparable in size to its major radius $R$, defined as
\begin{equation}
\label{9a}
R:=\frac{1}{|E|}\int_{E}r \dd (r,z).
\end{equation}
In contrast, bubble rings satisfying $a\ll R$ are called \emph{thin}. Independently from the precise definition of $a$, for theoretical purposes, determining a precise value for a threshold for the ratio $a/R$ up to which a ring may be called thin or from which on it may be called thick seems arbitrary. Instead, we propose a more geometric condition based on the ring's total mean-curvature.

\begin{definition}
\label{D1}
Let $E\subset \R^2_+$ be a convex symmetric set. Then the axisymmetric torus with cross-section $E$ is called a \emph{thick ring} if it has non-positive total mean-curvature, 
\begin{equation}
 \label{1}
\int_{\partial E} H\dd s\le 0.
\end{equation}
\end{definition}

The approximate geometry of a typical thick ring is plotted in Figure \ref{fig1}.\begin{figure}[t]
 \centering \includegraphics[scale=.5]{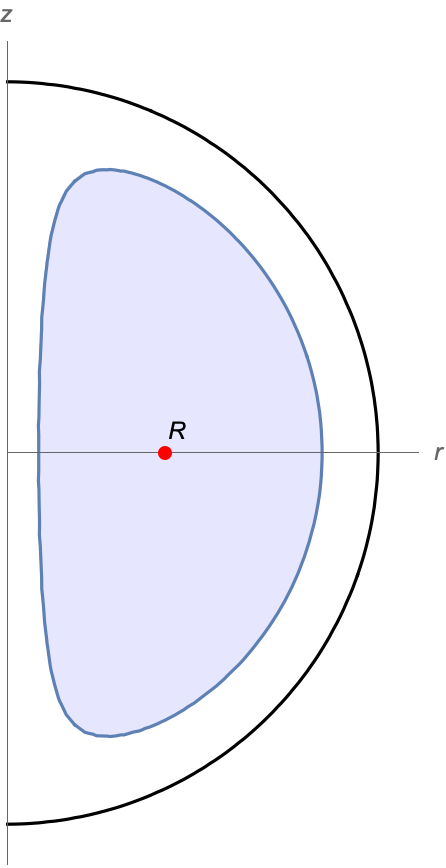}
	\caption{\label{fig1} A thick vortex inside of  a semi-circle of the same geometric center $(R,0)$.}
\end{figure}

In Lemma \ref{L1b} below, we see that the non-positivity condition \eqref{1} for the mean-curvature is equivalent to requiring that the vortex ring is placed sufficiently close to the symmetry axis, so that
\begin{equation}
    \label{21}
    \delta: = \int_{E}\frac{1}{r^2}\dd (r,z)-2\pi\ge0.
\end{equation}
For instance, for the thick vortex ring solutions to the ordinary Euler equations constructed by Norbury \cite{Norbury72}, which  approximate the toroidal stream surfaces of Hill’s spherical ring nearest the symmetry axis, so that $\mathrm{dist}(E,\partial \R^2_+)/R\ll1$, this thickness measure is very large, $\delta \sim R/\mathrm{dist}(E,\partial \R^2_+)\gg1$.

 Our main result establishes a lower bound on the Weber number if the bubble ring is thick, as defined in Definition \ref{D1}.

\begin{theorem}\label{T1} 
There exists a positive constant $C$ such that for any solution $(E, \psi, W, \gamma, \lambda)$ to the equations \eqref{4}, \eqref{3}, \eqref{5}, \eqref{6}, \eqref{7} that satisfies the curvature condition \eqref{1} and the maximum principle \eqref{47}, it follows that
\[
\We \ge \frac{C}{\mu+\mu^3}\left(\frac1{\mu^2}+\delta\right),
\]
where $\mu = R/\sqrt{|E|}$.
\end{theorem}

For any bubble ring whose cross-section $E$ is non-eccentric, the length scale $\sqrt{|E|}$ can be considered an approximation for the inner radius $a$ of the ring, and thus, for a typical thick bubble ring, we have $\mu \sim 1$. Figure \ref{fig1} illustrates this relationship: it shows a thick bubble ring inside a semi-circle of radius $3\pi R/4$, both having the same geometric centre $(R,0)$ and a comparable area. Evidently, if $\mu\sim 1$, the statement of the theorem becomes
\[
\We \ge C(1 + \delta).
\]
In particular, for any thick bubble ring in shape similar to Norbury’s vortex rings, we would necessarily have
\[
\We \ge \frac{CR}{\mathrm{dist}(E,\partial \R^2_+)}.
\]
This observation shows that Norbury’s perturbative construction cannot be carried out for bubble rings at a fixed Weber number.

We note the following consequence:

\begin{corollary}\label{C1}
There exists a positive constant $C$ such that for any solution $(E, \psi, W, \gamma, \lambda)$ to the equations \eqref{4}, \eqref{3}, \eqref{5}, \eqref{6}, and \eqref{7} that satisfies the maximum principle \eqref{47} and the bound
\begin{equation}
    \label{18}
2\pi R^2 \le |E| ,
\end{equation}
it follows that
\begin{equation}
\label{11}
\We \geq C(1+\delta).
\end{equation}
\end{corollary}

 Actually, the proof of this corollary relies on the observation that any ring for which \eqref{18} holds is thick in the sense of Definition \ref{D1}, see Lemma \ref{L1}. The reverse statement is wrong, as illustrated in the following example.

 \begin{example}\label{E1}
Consider a symmetric set $E$ that is enclosed by an ellipse,
\[
E = \left\{ (r,z)\in\R^2_+:\: \frac{(r-R)^2}{m^2} +\frac{z^2}{n^2}\le 1\right\}.
\]
Using the polar coordinates $r=R+mt\cos\theta$ and $z=nt\sin \theta$ with $t\in[0,1)$ and $\theta\in[-\pi,\pi)$, we can write the curvature integral as
\[
\int_E \frac1{r^2}\dd (r,z) = mn \int_0^1 \int_{-\pi}^{\pi}\frac{t}{(R +mt\cos \theta)^2}\dd\theta\dd t.
\]
A straightforward computation reveals that
\[
\int_{-\pi}^{\pi} \frac1{R +mt\cos\theta}\dd \theta = \frac{2\pi}{\sqrt{R^2 -m^2 t^2}},
\]
and thus
\[
\int_{-\pi}^{\pi} \frac1{(R+mt\cos\theta)^2}\dd\theta = -\frac{\dd}{\dd R} \int_{-\pi}^{\pi} \frac1{R+mt\cos\theta}\dd \theta = \frac{2\pi R}{(R^2-m^2 t^2)^{3/2}}.
\]
We deduce that
\[
\int_{E}\frac1{r^2}\dd (r,z) =  \int_0^1 \frac{2\pi mnR t}{(R^2-m^2 t^2)^{3/2}}\dd t=\frac{2\pi n}{m}\left(\frac{R}{\sqrt{R^2-m^2}}-1\right),
\]
and thus, the condition \eqref{21} is equivalent to
\begin{equation}
    \label{30}
\frac{m}n+1 \le \frac{R}{\sqrt{R^2-m^2}}.
\end{equation}

We suppose now that the left end point of the ellipse is given by $(\eps,0)$, so that $R = m+\eps$. Rewriting \eqref{30} then gives
\[
\frac{m}n +1  \le \frac{m+\eps}{\sqrt{\eps}\sqrt{2m+\eps}}.
\]
Hence, for any given axis ratio $m/n$, the condition \eqref{21} can be satisfied if the ellipse is moved sufficiently close to the $z$ axis, $\eps\ll1$.

 On the  other hand, using the well-known fact that the area of the ellipse is  $|E|=\pi mn$, we may rewrite \eqref{18} as
\[
2\left(\frac{m+\eps}{n}\right)^2 \le \frac{m}{n}.
\]
Hence, \eqref{18} is violated if the aspect ratio $m/n$ is large.
\end{example}

The proof of our main results   will be the content of the next section.

\section{Proofs} \label{S3}
 
 For notational simplicity, we introduce the approximate minor  radius of the toroidal bubble as
 \begin{equation}
     \label{23}
     a:=\left(\frac{|E|}{2\pi}\right)^{1/2},
 \end{equation}
 where \(|E|\) is the area of the cross section \(E\).  Our choice of constants is motivated by condition \eqref{18} in Corollary \ref{C1}, which now reads $R\le a$.

We start with two simple geometric observations.

\begin{lemma}
    \label{L1b}
There is the identity
\[
\int_{\partial E}H\dd s +\delta = 0,
\]
and thus, the conditions \eqref{1} and \eqref{21} are equivalent.
\end{lemma}

\begin{proof}
It suffices to decompose the mean curvature into the meridional curvature and the azimuthal curvature \eqref{20} and apply the Gauss--Bonnet theorem and the divergence theorem to the individual terms:
\begin{equation}\label{41}
\int_{\partial E} H\dd s = \int_{\partial E}\kappa \dd s +\int_{\partial E} \frac{n\cdot e_r}{r}\dd s = 2\pi -\int_E \frac1{r^2}\dd(r,z).
\end{equation}
The definition of $\delta$ in \eqref{21} yields the statement.
\end{proof}

 \begin{lemma}\label{L1}
   Suppose that \eqref{18} holds. Then \eqref{1} is true.
 \end{lemma}

 Obviously, Corollary \ref{C1} is a consequene of Theorem \ref{T1} and Lemma \ref{L1}.

\begin{proof} According to Lemma \ref{L1b}, it is enough to establish
  \eqref{21}. We use twice the  Cauchy--Schwarz inequality and obtain
    \[
  |E| \le R \int_{E}\frac{1}{r}\dd (r,z)  \le R \left(\int_{E}\frac{1}{r^2}\dd (r,z)\right)^{1/2}|E|^{1/2} .
    \]
    It immediately follows that
    \[
    \frac{|E|}{R^2}\le  \int_E\frac1{r^2}\dd (r,z),
    \]
    and thus, the bound in \eqref{18} yields \eqref{21}.
\end{proof}

At the heart of the proof of Theorem \ref{T1} are two crucial observations. The first connects the curvature sign condition \eqref{1} with the non-negativity of the Bernoulli constant.

\begin{lemma}
    \label{L2}
Let $\beta>0$ be given. Suppose that there exists a solution $(E,\psi,W, \gamma,\lambda)$ to \eqref{4}, \eqref{3}, \eqref{5}, \eqref{6}, \eqref{7}, and \eqref{10} satisfying the curvature condition \eqref{1} and the maximum principle \eqref{47}. Then the Bernoulli constant is non-negative,
more precisely,
\begin{equation}\label{99}
\lambda|\partial E|\ge 2\delta \sigma.
\end{equation}
\end{lemma}

Next, we give a condition that yields a bound on the Weber number.

\begin{proposition}\label{P1}
There exists a positive constant \(C\) such that the following is true: Let $\beta>0$ be given and suppose that there exists a solution $(E,\psi,W, \gamma,\lambda)$ to \eqref{4}, \eqref{3}, \eqref{5}, \eqref{6}, \eqref{7} that satisfies \eqref{10} and \eqref{47}.    Assume furthermore that \(\lambda\geq 0\).  Then
\[
\We\geq \frac{C}{\mu+\mu^3}\left(\frac1{\mu^2}+\delta \right),
\]
where $\mu= R/\sqrt{|E|}$.
\end{proposition}

Actually, the proof of our first theorem is now obvious.

\begin{proof}[Proof of Theorem \ref{T1}.]
The statement  follows directly from Lemma \ref{L2}, Proposition \ref{P1}, and the definition in \eqref{23}.
\end{proof}

Before turning to the   proofs of Lemma \ref{L2} and Proposition \ref{P1}, it is convenient to non-dimensionalise the problem. Setting
\[(r,z)= (a\hat{r},a\hat{z}), \quad \psi= a\beta \hat{\psi},\]
     the boundary conditions \eqref{3} and \eqref{5} become
\begin{align}
     \hat{\psi}& =\frac{\hat{W}}{2}\hat{r}^2+\hat{\gamma}\quad \mbox{on }\partial \hat E,\label{31}\\
     2\hat{H} +\hat{\lambda} & = \We \left(\frac{1}{\hat{r}}\partial_n\hat{\psi}-\hat{W}n\cdot e_r\right)^2\quad \mbox{on }\partial \hat E.\label{32}
\end{align}
where we have set
\[\hat{W}=\frac{Wa}{\beta}, \,\hat{\gamma}=\frac{\gamma}{a\beta}, \,\hat{\lambda}=\frac{\lambda a}{\sigma}.\]
Moreover, the circulation \eqref{10} turns into
\[
\int_{\partial \hat E} \frac1{\hat r} \partial_n\hat \psi \dd s=-1.
\]
Thus without loss of generality, we may assume from here on that \(\beta=1\) and \(|E|=2\pi\) or, equivalently, $a=1$. For notational convenience, we drop the hats in the following.

\begin{proof}[Proof of Lemma \ref{L2}]
We start by noting that after non-dimensionalisation, the factor $\sigma$ can be dropped in equation \eqref{99}.

   We integrate the non-dimensionalised jump condition \eqref{32} over the boundary \(\partial E\) and obtain 
\begin{align}
    2 \int_{ \partial E} H\dd s+\lambda |\partial E|&=\We\int_{\partial E}\left(\frac{1}{r}\partial_n\psi-Wn\cdot e_r\right)^2 \dd s\geq 0.\label{15}
\end{align}
Invoking the curvature sign condition \eqref{1} and the identity in Lemma \ref{L1b},  we conclude that  
\[
\lambda |\partial E|\ge 2\delta \ge0,
\]
as desired.
\end{proof}

We provide an auxiliary geometric result, which exploits the convexity of the bubble rings' cross-sections.

\begin{lemma}\label{L3}
    Let \(E\subset\mathbb{R}^2_+\) be a compact convex set with positive area.  Then
    \[
    \Rmax=\max\{r>0: (r,z)\in E\}\le  3R.
    \]
\end{lemma}
\begin{proof}
We set \(k=\Rmax /R.\)  It is clear that $k>1$ because $E$ has positive area. Our goal is to  show that \(k\le 3.\)

Consider the vertical line \(\{r=R\},\) which intersects twice with the set \(E.\) We denote these two intersection points by \(B_1\) and \(B_2\). Suppose the point \(A\) realises the maximal distance with respect to the axis \(\{r=\Rmax \}\).   By convexity, it is clear that the triangle $T_+$ generated by $A$, $ B_1$ and $B_2$ is contained in the set $E_+:=\{(r,z)\in E:\: r>R\}$, see  Figure \ref{fig2}.  \begin{figure}[t]
        \centering
        \includegraphics[scale=0.6]{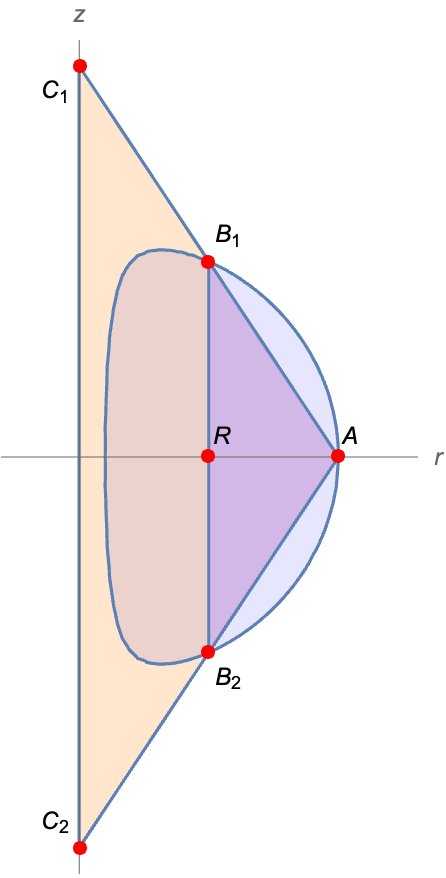}
        \caption{Construction of the triangle contained in $E_+$ and the trapezoid containing $E_-$ in the case of a symmetric convex set $E$.}
        \label{fig2}
    \end{figure}

Likewise, if we extend the segments \(\overline{AB_1}\) and \(\overline{AB_2}\) to straight lines and denote their respective intersection points with the $z$ axis by \(C_1\) and \(C_2\), then by convexity once again, we see that the set \(E_-=\{(r,z)\in E:\:  r\leq R\}\) is contained in the trapezoid $T_-$ generated by $B_1,B_2,C_2$, and $C_1$, see also Figure \ref{fig2}.

In order to estimate $k$, we start by observing that our definition of $R$ and the previous observations yield  
   \begin{equation}\label{a}
\begin{aligned}
\int_{T_-} (R-r) \dd (r,z) &\ge \int_{E_-} (R-r) \dd (r,z)\\
&=\int_{E_+}(r-R) \dd (r,z) \ge \int_{T_+}(r-R) \dd (r,z).
    \end{aligned}   \end{equation} 
The integrals on the both sides of the estimates can be computed explicitly.   For this, we define    \(L=\mathrm{dist}(B_1,B_2).\) Since the set $E$ has positive measure, we know that $L$ must be positive. For the integral on the right-hand side of \eqref{a}, we observe that
    \begin{align*}
        \int_{T_+}(r-R) \dd (r,z)
        &=\int_{R}^{kR}\left(r-R\right)\left[-\frac{Lr}{(k-1)R}+\frac{k}{k-1}L \right] \dd r\\
        &= LR^2\left(\frac{1}{6}k^2-\frac{1}{3}k+\frac{1}{6}\right)\\
        &=\frac{1}{6}LR^2(k-1)^2.
    \end{align*}
The integral on the left-hand side of \eqref{a} can be calculated similarly: 
    \begin{align*}
        \int_{T_-} (R-r) \dd (r,z) 
        &= \int_{0}^{R}(R-r)\left[\frac{k}{k-1}L-\frac{Lr}{(k-1)R}\right] \dd r\\
        &= LR^2\left[\frac{k}{2(k-1)}-\frac{1}{6(k-1)}\right].
    \end{align*}
    Inserting both identities into  \eqref{a} and using that $k>1$ gives
    \[
    (k-1)^3 \le 3k-1
    \]
    or, equivalently, $k\le 3$ as desired.
\end{proof}
\begin{remark}
It should be noticed that the constant 3 is sharp as the equality can be realised by a triangle symmetric in \(z\) whose one side is exactly on the \(z\) axis.
\end{remark}
Next, we turn to the proof of the key proposition.

\begin{proof}[Proof of Proposition \ref{P1}]
    For given \(z\in\R,\) let \(w(z)\) be the horizontal width of \(E,\) which is defined as
    \[w(z):=\Rmax (z)-\Rmin(z),\]
where
\[
\Rmin(z) = \min\{r:\:(r,z)\in\overline{E}\},\quad \Rmax(z) = \max\{r:\:(r,z)\in\overline{E}\}.\]
 We furthermore denote by
\[
h  =\sup\{z\in\R:\: w(z)>0\}
\]
the maximal height of the set $E$ over the $r$ axis. Since \(|E|=2\pi\) and because \(0\leq w(z)\le 3R\) by the virtue of Lemma \ref{L3}, we see that 
    \[2\pi=\int_{-h}^{h}w(z)\dd z\le 6R h,\]
which in turn implies   
    \begin{equation}
       2h\geq\frac{2\pi}{3R}\label{13} .
    \end{equation}

We now define
\[
\mathcal{S}(b)= \left\{(r,z)\in\partial E:\: n(r,z)\cdot e_r> b\right\} 
\]
for  $b\in(0,1)$.
Our goal is to show that there exists a positive \(b_0\) such that  \(\mathcal{S}(b)\) has positive length for all \(b\in(0,b_0).\) 
For $b=0$, the part $\mathcal{S}(0)$ of the boundary can be parametrised as a graph over  $z$,
    \[\gamma_{0}(z)=(f(z),z),\, z\in (-h,h).\]
This is actually possible thanks to the convexity of the shape.
The arc-length can be represented by 
    \[\dd s=\sqrt{1+(f'(z))^2}\,\dd z\]
    and we have
    \begin{equation}
        \label{35}
        n\cdot e_r=\frac{1}{\sqrt{1+(f'(z))^2}}
    \end{equation}
    for the radial part of the outer normal.
Therefore, we can compute the length of this part of the boundary as
    \begin{align*}
       |\mathcal{S}(0)|=\int_{-h}^{h}\sqrt{1+(f')^2}\,\dd z 
        \leq \int_{-h}^{h}(1+|f'|)\, \dd z
        =2h+\int_{-h}^{h}|f'| \,\dd z.
    \end{align*}
    Thanks to the convexity and symmetry of $E$, we have that
    \[
    \int_{-h}^h |f'|\dd z = 2\int_{-h}^0f'(z)\dd z = 2\left(f(0)-f(-h)\right) \le 2\Rmax.
    \]
    Invoking Lemma \ref{L3}, we then find
    \begin{equation}
        \label{36}
   |\mathcal{S}(0)|\le 2h + 6R.
    \end{equation}

We now use \eqref{35} to write and estimate
\[
2h = \int_{-h}^h \dd z = \int_{\mathcal{S}(0)} n\cdot e_r\dd s =  \int_{\mathcal{S}(b)} n\cdot e_r\dd s +  \int_{\mathcal{S}(0)\setminus \mathcal{S}(b)} n\cdot e_r\dd s
\le |\mathcal{S}(b)|  + b|\mathcal{S}(0)\setminus \mathcal{S}(b)|.
\]
Making use of  \eqref{36}, we get
\[
2h \le |\mathcal{S}(b)| +b\left(2h+6R - |\mathcal{S}(b)|\right) = b(2h+6R) +(1-b)|\mathcal{S}(b)|,
\]
and thus
\[
|\mathcal{S}(b)| \ge 2h-\frac{6b}{1-b}R.
\]
Plugging in the bound in \eqref{13} gives 
\[
|\mathcal{S}(b)| \ge \frac{2\pi}{3R } - \frac{6bR }{1-b} .
\]
If $R >\sqrt{2\pi}/6$, we choose $b= \pi/(36R ^2)<1$, and  obtain 
\begin{equation}
    \label{37}
    |\mathcal{S}(b)| \ge \frac{\pi}{3R }. 
\end{equation}
Otherwise, we choose $b=1/2$ and get
\[
|\mathcal{S}(b)|\ge \frac{2\pi}{3R } - 6R  =\frac{\pi}{2R }\left(\frac{4}{3}-\frac{12R ^2}{\pi}\right)\ge \frac{\pi}{3R }.
\]
Hence, we have the bound \eqref{37} in  either case.

To finally deduce a bound on the Weber number, we write the jump condition \eqref{32} as 
\[
\sqrt{2H+\lambda}=\sqrt{\We}\left(-\frac{1}{r}\partial_n\Psi\right),
\]
where we have used that   $\partial_n\Psi\le 0$ on $\partial E$ thanks to the maximum principle \eqref{47}.
Integrating over \(\partial E\) and using the fact that \(\beta=1,\) we obtain 
\[\sqrt{\We}  = \int_{\partial E}\sqrt{2H+\lambda}\dd s.\]
Since \(\lambda\geq 0\) by assumption and because the convexity of $E$ implies that the meridional curvature is non-negative, we estimate 
\begin{align*}
   \sqrt{\We}&\geq \frac1{\sqrt{2}} \int_{\mathcal{S}(b)}\sqrt{\frac{2n\cdot e_r}{r}}\dd s+\sqrt{\frac{\lambda}2}|\mathcal{S}(0)|\\&  \geq \sqrt{\frac{b}{\Rmax}} |\mathcal{S}(b)|+\sqrt{2\lambda}h  \geq \sqrt{\frac{b}{3R}}\frac{\pi}{3R }+\sqrt{2\lambda}h.
\end{align*}
In the last inequality we have used Lemma \ref{L3} and \eqref{37}. Choosing the precise values of $b$ as above, the latter implies
\begin{equation}
    \label{101}
    \We \ge C\left(\frac1{R^5+R^3}+\lambda h^2\right),
\end{equation}
for some constant $C>0$.

Now, invoking Lemma \ref{L2} and brutally estimating the perimeter of $E$ by that of the enclosing rectangle, we see that 
\[
    {\lambda}h^2 \ge  \frac{\delta h^2}{|\partial E|}\ge\frac{\delta h^2}{4h+2\Delta R},
    \]
    where we have set $\Delta R = \Rmax-\Rmin$. The expression on the right-hand side is a monotone function of $h$. By estimating the area of $E$ by the enclosing rectangle, we find $\pi \le h\Delta R$, and thus the previous estimate becomes
    \[
    \lambda h^2 \ge \frac{\delta \pi^2}{\Delta R\left(4\pi +2(\Delta R)^2\right)}
    \]
    Thanks to Lemma \ref{L3}, we may  estimate $\Delta R\le 3R$. Plugging the resulting bound into \eqref{101}, we deduce that 
    \[
    \We \ge \frac{C}{R+R^3}\left(\frac1{R^2}+\delta\right),
    \]
    for some new constant $C$.
\end{proof}

\section*{Acknowledgment}
The authors acknowledge stimulating discussion with Lukas Niebel. They thank the anonymous referees for suggestions.

\bibliography{references.bib}
\bibliographystyle{abbrv}
\end{document}